\documentclass[12pt]{amsart}
\usepackage[margin=1.in, centering]{geometry}

\usepackage{amssymb,latexsym,amsmath,extarrows, mathrsfs, amsthm}
\usepackage{graphicx}

\usepackage[colorlinks, linkcolor=blue]{hyperref}

\usepackage{color}
\usepackage{wasysym}

\usepackage{float}

\newtheorem{theorem}{Theorem}[section]
\newtheorem{lemma}[theorem]{Lemma}
\newtheorem{proposition}[theorem]{Proposition}

\newtheorem{definition}[theorem]{Definition}
\newtheorem{corollary}[theorem]{Corollary}

\theoremstyle{remark}
\newtheorem{remark}{\bf  \itshape  Remark}

\newcommand{\bgcd}{{\rm gcd}_{{\bf b}}}

\newcommand{\vb}{{{\bf b}}}

\allowdisplaybreaks

\title{Random walks on generalized visible lattice points}

\author{Kui Liu}
\address{School of Mathematics and Statistics, Qingdao University, 308 Ningxia Road, Shinan District, Qingdao, Shandong, China}
\email{liukui@qdu.edu.cn}

\author{Xianchang Meng}
\address{Mathematisches Institut,
	Georg-August Universit\"{a}t G\"{o}ttingen,
	Bunsenstra{\ss}e 3-5,
	D-37073 G\"{o}ttingen,
	Germany}
\email{xianchang.meng@uni-goettingen.de}


\begin{document}

\maketitle

\begin{abstract}
	We consider the proportion of generalized visible lattice points in the plane visited by random walkers. Our work concerns the visible lattice points in random walks in three aspects: (1) generalized visibility along curves; (2) one random walker visible from multiple watchpoints; (3) simultaneous visibility of multiple random walkers. Moreover, we found new phenomenon in the case of multiple random walkers: for visibility along a large class of curves and for any number of random walkers, the proportion of steps at which all random walkers are visible simultaneously is almost surely larger than a positive constant. 
\end{abstract}

\section{Introduction}

\subsection{Background}

In the two-dimensional integer lattice $\mathbb{Z}^2$, an lattice point $P\in\mathbb{Z}^2$ is said to be visible (from the origin) if there is no other lattice point on the straight line segment joining the origin and $P$. A classical result proved by Sylvester \cite{S} in 1883 indicates that the density of visible lattice points in $\mathbb{Z}^2$ is $1/\zeta(2)=6/\pi^2\approx 0.60793$, where $\zeta(s)$ is the Riemann zeta function.
 
Fix ${\bf b}=(b_1,b_2)\in\mathbb{N}^2$ with $\gcd(b_1,b_2)=1$. We define the following generalized visibility of lattice points along curves.

\begin{definition} \label{def: b-visibility}
Given two distinct lattice points $P=(p_1,p_2)$ and $Q=(q_1,q_2)$ in $\mathbb{Z}^2$, they determine a curve  joining $P$ and $Q$ of type $a_1(y-q_2)^{b_1}=a_2(x -q_1)^{b_2}$ for some $(a_1,a_2)\in\mathbb{Q}^2\setminus \{(0,0)\}$. If there is no other lattice point lying on the segment of this curve between $P$ and $Q$, then we say $P$ is ${\bf b}$-visible from $Q$.
\end{definition}

Note that the ${\bf b}$-visibility is mutual. Indeed, if $P$ is ${\bf b}$-visible from $Q$ along the curve $a_1(y-q_2)^{b_1}=a_2(x-q_1)^{b_2}$, then $Q$ is also ${\bf b}$-visible from $P$ along the curve $a_1(y-p_2)^{b_1}=(-1)^{b_1+b_2}a_2(x-p_1)^{b_2}$. 

If an lattice point $P\neq (0,0)$ is ${\bf b}$-visible from the origin, we say it is ${\bf b}$-\textit{visible} for short. There are some previous works on the density of ${\bf b}$-visible points from the origin, see e.g. \cite{BEH} and \cite{GHKM}. Recently, Liu and Meng \cite{LM} proved precise asymptotic formulas for the number of lattice points $(1,b_2)$-visible from multiple watchpoints simultaneously in square areas.

In 2015, Cilleruelo, Fern\'{a}ndez and Fern\'{a}ndez \cite{CFF} studied finer structure of the set of visible (i.e. ${\vb}=(1,1)$) lattice points from the view point of random walks. For $0<\alpha<1$, an $\alpha$-random walk starting at the origin on $\mathbb{Z}^2$ is defined by
\begin{equation}\label{eq: random-walk-pattern}
P_{i+1}=P_i+\begin{cases}
(1,0), &\ {\rm with\ probability\ }\alpha,\\
(0,1), &\ {\rm with\ probability\ }1-\alpha,
\end{cases}
\end{equation}
where $P_i=(x_i, y_i)$ is the coordinate of the $i$-th step of the $\alpha$-random walk for $i=0,1,2,\cdots,$ and $P_0=(0,0)$. They showed that the proportion of such type of visible points visited by an $\alpha$-random walk is almost surely $1/\zeta(2)$.

In this paper, combining probabilistic arguments and tools from analytic number theory, we generalize the result in \cite{CFF} in three aspects: (1) ${\vb}$-visible lattice points; (2) one random walker visible from multiple watchpoints; (3) simultaneous $\vb$-visibility of multiple random walkers.

\subsection{Our results} Throughout this paper, we always assume ${\bf b}=(b_1,b_2)\in\mathbb{N}^2$ is fixed with $\gcd(b_1,b_2)=1$.

We first consider the $\vb$-visibility of one random walker from multiple watchpoints. Suppose the watchpoints set 
$$
\mathcal{W}:=\big\{(u_j,v_j)\in\mathbb{Z}^2,~1\leq j\leq J\big\}
$$
satisfies condition
\begin{center}
$(*)$\quad distinct lattice points in  $\mathcal{W}$ are pairwise ${\vb}$-visible from each other.
\end{center}
We remark that the cardinality $J=|\mathcal{W}|$ can't be very large according to Lemma \ref{lem: cardinality J<=} in the next section. If an lattice point is $\bf b$-visible from all points in $\mathcal{W}$, we say it is ${\bf b}$-\textit{visible from} $\mathcal{W}$. 

Given an $\alpha$-random walk defined in \eqref{eq: random-walk-pattern}, consider a sequence of random variables associated with $\mathcal{W}$
\begin{equation}\label{eq: expression X(W)_i}
X_i:=X(\vb, \mathcal{W})_i=\begin{cases}
1,&\text{if $P_i$ is~{\bf b}-visible from}~\mathcal{W},\\
0,&\text{otherwise},
\end{cases}
\end{equation}
for $i=1,2,\cdots$. Then the random variable
\begin{equation}\label{eq: }
\overline{S}(n):=\overline{S}(n)_{\alpha, \vb, \mathcal{W}}=\frac{1}{n}\sum_{1\leq i\leq n} X_i
\end{equation}
indicates the proportion of steps at which the $\alpha$-random walker is $\bf{b}$-visible from $\mathcal{W}$ in the first $n$ steps.

\begin{theorem}\label{thm: Theorem for multiple watchpoints}
For watchpoints set $\mathcal{W}$ satisfying condition $(*)$ with cardinality $J=|\mathcal{W}|<2^{b_1+b_2}$, we have
$$
\lim\limits_{n\rightarrow\infty}\overline{S}(n)=\prod\limits_{p}\bigg(1-\frac{J}{p^{b_1+b_2}}\bigg)
$$
almost surely, where $p$ runs over all primes.
\end{theorem}

\begin{remark}
	The right hand side is independent of $\alpha$. By Lemma \ref{lem: cardinality J<=}, if $|\mathcal{W}|=2^{b_1+b_2}$, there is no lattice point {\vb}-visible from $\mathcal{W}$, in which case $\lim_{n\rightarrow\infty}\overline{S}(n)=0$.
\end{remark}

We also consider multiple random walkers simultaneously {\bf b}-visible from the origin. Let $r\geq 1$ be an integer and
$$
\boldsymbol{\alpha}:=(\alpha_1,\cdots, \alpha_r)
$$ 
with $0<\alpha_1,\cdots,\alpha_r<1$. Similarly as in \eqref{eq: random-walk-pattern}, for $1\leq j\leq r$, we define the $\alpha_j$-random walk starting from the origin by
$$
\label{eq: multi-random-walk}
P^{(j)}_{i+1}=P^{(j)}_i+\begin{cases}
(1,0), &\ {\rm with\ probability\ }\alpha_j,\\
(0,1), &\ {\rm with\ probability\ }1-\alpha_j,
\end{cases}
$$
where $P^{(j)}_i$ is the $i$-th ($i=0,1,2,\cdots$) step of the $\alpha_j$-random walk and $P^{(j)}_0=(0,0)$. Define a sequence of random variables associated with $r$ walkers by
\begin{equation}\label{eq: expression Y(alpha)}
Y_i:=Y(\vb,r)_i=\begin{cases}
1,&\text{ if all   $P^{(j)}_i$ for $1\leq j\leq r$ are  $\mathbf{b}$-visible} ,\\
0,&\ {\rm otherwise}.
\end{cases}
\end{equation}
Then the random variable
\begin{equation*}
    \overline{R}(n):=\overline{R}(n)_{\boldsymbol{\alpha},{\bf b},r}=\frac{1}{n}\sum_{1\leq i\leq n} Y_i
\end{equation*}
indicates the proportion of steps at which all these $r$ walkers are $\bf{b}$-visible simultaneously in the first $n$ steps. 
\begin{theorem}\label{thm: Theorem for multiple walkers}
Assume $\vb=(b_1, b_2)$ with $\gcd(b_1, b_2)=1$. 
We have
$$
\lim\limits_{n\rightarrow\infty}\overline{R}(n)=\prod_{p}\bigg( 1-\frac{1}{p^{b_{\star}}}+\frac{1}{p^{b_{\star}}}\bigg(1-\frac{1}{p^{b^{\star}}}\bigg)^r \bigg)
$$
almost surely, where  $b_{\star}=\min(b_1, b_2)$, $b^{\star}=\max(b_1, b_2)$, and $p$ runs over all primes, 
\end{theorem}

\begin{remark}\label{remark-b1-b2}
In our proof, we only need to prove the case $b_1\leq b_2$. For the case $b_1\geq b_2$, we just switch the axes of the coordinate system for the lattice and replace $\vb$ by $\vb'=(b_2, b_1)$ in our results.
\end{remark}

Before doing careful calculations,   one may expect that the density of visible steps goes to $0$ as the number of random walkers increases. This is true if $b_{\star}=1$. However, when $b_{\star}\geq 2$,  the density in Theorem \ref{thm: Theorem for multiple walkers} approaches a positive constant far from $0$ density as $r$ goes to $\infty$. This is a surprisingly new phenomenon. 

\begin{corollary}
Assume $\vb=(b_1, b_2)$ with $\gcd(b_1, b_2)=1$ and $b_{\star}=\min(b_1, b_2)\geq 2$. For any number of random walkers, the proportion of steps at which all walkers are simultaneously $\vb$-visible from the origin is almost surely
$$\geq \frac{1}{\zeta(b_{\star})}\geq \frac{1}{\zeta(2)}=0.6079\ldots,$$
where $\zeta(s)$ is the Riemann zeta-function.
\end{corollary}

\subsection{Numerical experiments}
We verify the results in our theorems by doing numerical experiments using random generator in Python.

For one random walker $\vb$-visible from watchpoint set $\{ (0,0), (1, 2), (2,1) \} $, we calculate the proportion of $\vb$-visible steps in $\alpha$-random walks within $100, 000$ steps. Since our results are "almost surely", we did the same calculation $10$ times then take the average. We list densities for some values of $\vb$ with $\alpha=0.5$ and $\alpha=0.3$ in Table \ref{table-one-walker-mutiset}.

\begin{table}[htpb]
    \centering
    \begin{tabular}{|c|c|c|c|}
    \hline
       $\vb$  &   Numerical $\alpha=0.5$ &  Numerical $\alpha=0.3$ & Theoretical \\
        \hline
       $(1,2)$ & 0.534592 & 0.535673 & 0.534567 \\
       $(1,3)$ & 0.777417 &  0.777271 &  0.777373 \\
       $(1,4)$ & 0.894337  & 0.894110 & 0.894015 \\
       $(1, 5)$ & 0.949046 & 0.948974   & 0.948994 \\
       $(2, 3)$ & 0.894215 & 0.893807  & 0.894015 \\
       $(2, 5)$ & 0.975107  & 0.975023 &    0.975182 \\
       $(3, 4)$ & 0.975401  & 0.975118  & 0.975182 \\
       $(3, 5)$ &  0.987746  & 0.987856   & 0.987821 \\
       \hline
    \end{tabular}
    \caption{Density of $\vb$-visible steps  from multiple points}
    \label{table-one-walker-mutiset}
\end{table}

For multiple random walkers, we calculate proportions of $\vb$-visible steps for different number of walkers. We choose same $\alpha$-random walk for all walkers with $\alpha=0.5$, then calculate the number of $\vb$-visible steps within $100, 000$ steps and average the same process over $10$ times. See Table \ref{table-multi-walker} for a list of these densities with $\vb=(2,3)$ and $\vb=(3,5)$.

\begin{table}[htpb]
    \centering
    \begin{tabular}{|c|cc|cc|}
    \hline 
   & \multicolumn{2}{c|}{$\vb=(2,3)$} & \multicolumn{2}{|c|}{$\vb=(3,5)$} \\
    \hline
      $r$ & Numerical  & Theoretical & Numerical  & Theoretical\\
      \hline
       2 & 0.933192 & 0.933076 & 0.991948 & 0.992002 \\
       3 & 0.905554 & 0.905515 & 0.988254 & 0.988185 \\
       4 & 0.881253 & 0.881225 & 0.984412 & 0.984484 \\
       5 & 0.859504 & 0.859791 & 0.980724 & 0.980896 \\
       6 & 0.840929 & 0.840850 & 0.977486 & 0.977417 \\
%       \vdots & \vdots & \vdots\\
%       7 & 0.824151 & 0.824087 & 0.973835 & 0.974044 \\
%       8 & 0.809092 & 0.809229 & 0.971117 & 0.970773 \\
%       9 & 0.796485 & 0.796036 & 0.967437 & 0.967601 \\
       10 & 0.784229 & 0.784303 & 0.964403 & 0.964525 \\
       20 & 0.716891 & 0.716860 & 0.938527 & 0.938432 \\
       30 & 0.690567 & 0.690364 & 0.919002 & 0.919169 \\
       40 & 0.676851 & 0.676832 & 0.904943 & 0.904881 \\
       50 & 0.668499 & 0.668389 & 0.894245 & 0.894220 \\
       60 & 0.662547 & 0.662484 & 0.886117 & 0.886205 \\
%       70 & 0.658016 & 0.658092 & 0.880129 & 0.880125 \\
%       80 & 0.654760 & 0.654698 & 0.875562 & 0.875460 \\
%       90 & 0.651816 & 0.651996 & 0.871789 & 0.871834 \\
%       \vdots & \vdots & \vdots \\
       100 & 0.649852 & 0.649786 & 0.868927 & 0.868973 \\
       200 & 0.638448 & 0.638324 & 0.856645 & 0.856556 \\
%       300 & 0.633065 & 0.632951 & 0.851477 & 0.851337 \\
%       400 & 0.629946 & 0.629742 & 0.848071 & 0.847943 \\
%       \vdots & \vdots & \vdots \\
       500 & 0.628067 & 0.627636 & 0.845837 & 0.845638 \\
       1000 & 0.623631 & 0.622756 & 0.841814 & 0.841122 \\
       \vdots & \vdots & \vdots & \vdots & \vdots \\
       \hline
       $\infty$ & \multicolumn{2}{|c|}{ ${1}/{\zeta(2)}=0.607927\ldots$} &  \multicolumn{2}{|c|}{ ${1}/{\zeta(3)}=0.831907\ldots$} \\
      \hline
    \end{tabular}
    \caption{Density of visible steps for multiple walkers}
    \label{table-multi-walker}
\end{table}

 In our calculations, the numerical results match our theoretical results very well for at least two decimal digits and three decimal digits most of the time.

We also observe that in the numerical calculations within $100,000$ number of steps, the numerical results tend to be a little bit larger than the theoretical density as the number of walkers increases. Because in this situation the number of steps becomes not large enough relative to the number of walkers, these walkers have not fully spread out within such a fixed number of steps. In particular when some walkers meet each other at some point, the effective number of walkers we observe is actually smaller than the actual number and hence the numerical density tends to be a bit larger. The real 
situation might be much more complicated than this easy analysis, maybe that's why we didn't observe huge difference.

\bigskip
\textbf{Notations.} We use $\mathbb{Z}$ and $\mathbb{N}$ to denote sets of integers and positive integers, respectively. We also use the expressions $f=O(g)$ to
mean $|f|\leq Cg$ for some constant $C>0$. When the constant $C$ depends on some parameters ${\bf \rho}$, we write $f=O_{\bf \rho}(g)$. As usual, we use $[x]$ to denote the largest integer not exceeding the real number $x$ and use $|\mathcal{W}|$ to denote the cardinality of the set $\mathcal{W}$.

\bigskip

\textbf{Acknowledgements.} The first author is partially supported by Shandong Provincial Natural Science Foundation (Grant No. ZR2019BA028). The second listed author is partially supported by the Humboldt Professorship of Harald Helfgott.

\section{Preliminaries}

\subsection{Criteria for {\bf b}-visibility}

For integers $m,n\in\mathbb{Z}$ which are not both zero, we define 
$$
\bgcd(m,n):=\max\{d\in\mathbb{N}:d^{b_1}\mid m,~d^{b_2}\mid n\}.
$$
This generalized $\gcd$-function is bi-multiplicative and determined by its values on prime powers. Note that for any prime $p$ and any integers $k_1,k_2\geq 0$, we have
$$
\bgcd(p^{k_1},p^{k_2})=p^{\min\{[k_1/b_1],~[k_2/b_2]\}}.
$$
If $b_1=b_2=1$, then $\bgcd(m,n)$ is the classical greatest common divisor of $m$ and $n$. By elementary arguments, we derive the following lemma. Here we omit the proof.
\begin{lemma}
The $\bgcd$-function has the following properties.
\label{lem: properties of bgcd}
\begin{itemize}
\item[(i)]
Suppose $d\in\mathbb{N}$ and integers $m,n$ are not both zero, then $d\mid\bgcd(m,n)$ if and only if $d^{b_1}\mid m$ and $d^{b_2}\mid n$.
\item[(ii)] Suppose $b_1\leq b_2$, then for any integers $m,n$ which are not both zero, we have $\bgcd(m,n)=\bgcd(m+an,n)$ for any $a\in\mathbb{Z}$.
\end{itemize}
\end{lemma}

By a similar proof as Corollary $1$ of \cite{BEH}, we have the following criteria for {\vb}-visiblity.

\begin{lemma}\label{lem: criteria for b-visiblity}
Suppose lattice points $P=(p_1,p_2)\in\mathbb{Z}^2$ and $Q=(q_1,q_2)\in\mathbb{Z}^2$ satisfy $p_1\neq q_1$ and $p_2\neq q_2$, then $P$ is ${\bf b}$-visible from $Q$ if and only if $\bgcd(p_1-q_1,p_2-q_2)=1$.
\end{lemma}

\begin{remark}\label{remark-equal-coordinate}
	We remark that for lattice points $P=(p_1,p_2)$ and $Q=(q_1,q_2)$, if $p_1=q_1$, then the curve joining them in Definition \ref{def: b-visibility} degenerates to a vertical line and $Q$ is {\vb}-visible from $P$ if and only $q_2-p_2=\pm 1$. The $p_2=q_2$ case is similar.
\end{remark}

We note that the cardinality of watchpoints set satisfying condition $(*)$ can't be very large.
\begin{lemma}\label{lem: cardinality J<=}
	If a watchpoints set $\mathcal{W}$ satisfies condition $(*)$, then its cardinality $|\mathcal{W}|\leq 2^{b_1+b_2}$.
\end{lemma}

\begin{proof}
	Consider the map
	$$
	\lambda:\mathcal{W}\rightarrow\widetilde{\mathcal{W}}:=\{(u\bmod 2^{b_1}, v\bmod 2^{b_2})\mid (u,v)\in\mathcal{W}\}.
	$$
	Observe that the cardinality of $\widetilde{\mathcal{W}}$ is at most $2^{b_1+b_2}$. Hence, if $\widetilde{\mathcal{W}}$ has more than $2^{b_1+b_2}$ points, there must exist two
	distinct points $(u_1,v_1),(u_2,v_2)\in\mathcal{W}$ such that $\lambda((u_1,v_1))=\lambda((u_2,v_2))$. Thus we have
	$$
	2^{b_1}\mid (u_1-u_2)~\text{and}~2^{b_2}\mid (u_2-v_2),
	$$
	which implies $\bgcd(u_1-u_2,v_1-v_2)\geq 2$ by (i) of Lemma \ref{lem: properties of bgcd}. This contradicts our assumption on $\mathcal{W}$.
\end{proof}

\subsection{Useful lemmas}

To prove our theorems, we need the following key lemma from probability, which is essentially the second moment method.
\begin{lemma}[\cite{CFF}, Lemma 2.5]\label{lem: second moment method}
Let $X_i$, $i=1,2,\cdots$ be a sequence of uniformly bounded random variables such that
$$
\lim\limits_{n\rightarrow\infty}\mathbb{E}(\overline{S}_n)=\mu,
$$
where
$$
\overline{S}_n=\frac{1}{n}\sum\limits_{1\leq i\leq n}X_i.
$$
If there exists a constant $\delta>0$ such that the variance  $\mathbb{V}(\overline{S}_n)=O(n^{-\delta})$ for $n\geq 1$, then we have
$$
\lim\limits_{n\rightarrow\infty}\overline{S}_n=\mu
$$
almost surely.
\end{lemma}

The following binomial theorem with a congurence condition is also important.
\begin{lemma}[\cite{CFF}, Lemma 2.1]\label{lem: sum with congruence condition}
Suppose $0<\alpha<1$ and $n\in\mathbb{N}$. For any integer $1\leq d \leq n$ and any $
a\in\{0, 1,\cdots, d-1\}$, we have
$$
\sum\limits_{\substack{0\leq k\leq n\\k\equiv a\bmod d}}{n\choose k}
\alpha^k(1-\alpha)^{n-k}=\frac{1}{d}+O_{\alpha}(n^{-1/2})
$$
as $n\rightarrow\infty$.
\end{lemma}

The following Lemmas \ref{lem: sum with gcdb condition}-\ref{lem: mean of f_b^r(n)} are main ingredients for proofs of our theorems. We prove them in Section \ref{section-lemma-proofs}.

\begin{lemma}\label{lem: sum with gcdb condition}
Suppose $m,J\in\mathbb{N}$, $0<\alpha<1$, and vectors ${\bf s}=(s_1,\cdots,s_J)\in\mathbb{Z}^J$ and ${\bf t}=(t_1,\cdots,t_J)\in\mathbb{Z}^J$ satisfy $\bgcd(s_{j_1}-s_{j_2},t_{j_1}-t_{j_2})=1$ for any $1\leq j_1\neq j_2\leq J$. For $n>\max\limits_{1\leq j\leq J}|s_j|$ and any $\varepsilon>0$, we have
$$
\sum\limits_{\substack{0\leq k\leq m\\ \bgcd(n-s_j,k-t_j)=1,1\leq j\leq J}}{m\choose k}
\alpha^k(1-\alpha)^{m-k}=f_{{\bf b},{\bf s}}(n)+O_{\alpha,J,\varepsilon}(n^{\varepsilon}{m}^{-1/2}),
$$
where
\begin{align}\label{eq: expression of f_bu(n)}
f_{{\bf b},{\bf s}}(n)=\sum_{\substack{d_j^{b_1}\mid n-s_j,1\leq j\leq J\\ \gcd(d_{j_1},d_{j_2})=1,\forall 1\leq j_1\neq j_2\leq J}}\frac{\mu(d_1)\cdots\mu(d_J)}{(d_1\cdots d_J)^{b_2}}
\end{align}
and $\mu$ is the M{\"o}bius function. 
\end{lemma}

The mean value of $f_{{\bf b},{\bf s}}(n)$ plays an important role in our computations of expectations and variances of $\overline{S}_n$ and $\overline{R}_n$. With the help of Lemma \ref{lem: sum with gcdb condition}, we derive the following result.

\begin{lemma}\label{lem: mean of f_bu}
Let $f_{{\bf b},{\bf s}}(n)$ be given by \eqref{eq: expression of f_bu(n)}. If $b_1\leq b_2$, then for $x\geq 2$ we have
$$
\sum\limits_{n\leq x}f_{{\bf b},{\bf s}}(n)=x\prod\limits_{p}\bigg(1-\frac{J}{p^{b_1+b_2}}\bigg)+O_{{\bf b},{\bf s}}\big(\log^{J} x\big).
$$
\end{lemma}

If $J=1$ and ${\bf s}=s_1=0$, then the function $f_{\vb,{\bf s}}(n)$ is multiplicative, in which case we denote it by
$$
f_{{\bf b}}(n):=\sum\limits_{d^{b_1}\mid n}\frac{\mu(d)}{d^{b_2}}.
$$
For any prime $p$ and any integer $k\geq 1$, the above definition gives
\begin{equation}
f_{{\bf b}}(p^k)=
\begin{cases}
1,& \text{if}~ 1\leq k< b_1,\\
1-p^{-b_2},& \text{if}~k\ge b_1.
\end{cases}
\end{equation}
It follows that $0<f_{{\bf b}}(n)\leq 1$ for any $n\in\mathbb{N}$.

For the case of multiple random walkers, the following higher moments of $f_{{\bf b}}(n)$ is needed to prove Theorem \ref{thm: Theorem for multiple walkers}.
\begin{lemma}\label{lem: mean of f_b^r(n)}
Let $r\geq 1$ be an integer. For $x\geq 2$ and any $\varepsilon>0$, we have
$$
\sum\limits_{n\leq x}f_{\bf b}(n)^r=C_{{\bf b},r}x+O_{\varepsilon,{\bf b},r}(x^{\frac{1}{2}+\varepsilon}),
$$
where the constant
$$
C_{{\bf b},r}=\prod_{p}\bigg( 1-\frac{1}{p^{b_1}}+\frac{1}{p^{b_1}}\Big(1-\frac{1}{p^{b_2}}\Big)^r \bigg)
$$
with $p$ running over all primes.
\end{lemma}
In our proofs, we use the following estimates several times. We state them here without a proof.
\begin{lemma}\label{lem: bounds for sum over 1<i<i'<n}
	We have
	$$
	\sum\limits_{1\leq i<i^{\prime}\leq n}i^{-1/2}=O(n^{3/2})\indent\text{and}
	\indent\sum\limits_{1\leq i<i^{\prime}\leq n}(i^{\prime}-i)^{-1/2}=O(n^{3/2})
	$$
	as $n\rightarrow\infty$.
\end{lemma}

\section{Proof of Theorem \ref{thm: Theorem for multiple watchpoints}}

By Lemma \ref{lem: second moment method}, we only need to compute the expectation and variance of $\overline{S}(n)$.

\begin{proposition}\label{prop: (E(S(n))_a,W=}
If $b_1\leq b_2$, then for any $\varepsilon>0$ we have
\begin{align}
\mathbb{E}\big(\overline{S}(n)\big)=\prod\limits_{p}\bigg(1-\frac{J}{p^{b_1+b_2}}\bigg)+O_{\mathcal{W},\varepsilon}(n^{-1/2+\varepsilon}),
\end{align}
where $p$ runs over all primes.
\end{proposition}

\begin{proof}
By the definition of $X_i$, we have
$$
\mathbb{E}\big(\overline{S}(n)\big)=\frac{1}{n}\sum_{1\leq i\leq n} \mathbb{E}\big(X_i\big)=\sum_{1\leq i\leq n} \mathbb{P}\big(P_i~\text{is {\bf b}-visible from~ } \mathcal{W}\big).
$$
Let $n_0>0$ be a constant larger than $\max\limits_{1\leq i\leq J}(|u_j|+1)$ and $\max\limits_{1\leq i\leq J}(|v_j|+1)$, then
\begin{align}\label{eq: E(S(n))=}
\mathbb{E}\big(\overline{S}(n)\big)=\frac{1}{n}\sum_{n_0<i\leq n} \mathbb{P}\big(P_i~\text{is {\bf b}-visible from~}\mathcal{W}\big)+O_{\mathcal{W}}(n^{-1}).
\end{align}
For simplicity, we denote
\begin{align}\label{eq: definition of P_1^i}
\mathbb{P}_i^1:=\mathbb{P}(P_i~\text{is {\bf b}-visible from~}\mathcal{W})
\end{align}
For an $\alpha$-random walker, the coordinate of $P_i$ at the $i$-th step can be written as $P_i=(k,i-k)$ for some $k=0,1,\cdots,i$. The probability that $P_i$ is of the form $(k, i-k)$ is $\binom{i}{k}
\alpha^k (1-\alpha)^{i-k}$, which implies
\begin{align}\label{eq: one-point-probability}
\mathbb{P}_i^1:=\sum\limits_{\substack{0\leq k\leq i\\ (k,i-k)~\text{is {\bf b}-visible from}~\mathcal{W}}}{i \choose k}
\alpha^k(1-\alpha)^{i-k}.
\end{align}
By Remark \ref{remark-equal-coordinate} in Section 2, we see that for $i>n_0$, points $P_i=(k,i-k),~0\leq k\leq i$ are either not $\bf b$-visible from $\mathcal{W}$ or $k\neq u_j$ and $i-k\neq v_j$ for $1\leq j\leq J$. It then follows from Lemma \ref{lem: criteria for b-visiblity} that
$$
\mathbb{P}_i^1=\sum\limits_{\substack{0\leq k\leq i\\ \bgcd(k-u_j,i-k-v_j)=1,1\leq j\leq J}}{i \choose k}
\alpha^k(1-\alpha)^{i-k}
$$
By (ii) of Lemma \ref{lem: properties of bgcd}, we have
$$
\mathbb{P}_i^1=\sum\limits_{\substack{0\leq k\leq i\\ \bgcd(i-u_j-v_j,i-k-v_j)=1,1\leq j\leq J}}{i \choose k}
\alpha^k(1-\alpha)^{i-k}
$$
Note that distinct lattice points in $\mathcal{W}$ are pairwise ${\bf b}$-visible from each other, then by Lemma \ref{lem: criteria for b-visiblity} and (ii) of Lemma \ref{lem: properties of bgcd}, there holds
$$
\bgcd(u_{j_1}+v_{j_1}-u_{j_2}-v_{j_2},v_{j_1}-v_{j_2})=\bgcd(u_{j_1}-u_{j_2},v_{j_1}-v_{j_2})=1
$$ 
for any $1\leq j_1\neq j_2\leq J$. Thus by Lemma \ref{lem: sum with gcdb condition}, for $n_0<i\leq n$ and any $\varepsilon>0$, there holds
\begin{align}\label{eq: sum_k i choose k=}
\mathbb{P}_i^1=f_{{\bf b},{\bf u}+{\bf v}}(i)+O_{\alpha,\mathcal{W},\varepsilon}(i^{-1/2+\varepsilon}),
\end{align}
where $f_{{\bf b},{\bf u}+{\bf v}}$ is given by \eqref{eq: expression of f_bu(n)} with ${\bf s}={\bf u}+{\bf v}$. Combining this with \eqref{eq: E(S(n))=} and \eqref{eq: definition of P_1^i}, we obtain
$$
\mathbb{E}\big(\overline{S}(n)\big)=\frac{1}{n}\sum\limits_{n_0\leq i\leq n}f_{{\bf b},{\bf u}+{\bf v}}(i)+O_{\alpha,\mathcal{W},\varepsilon}\Big(\frac{1}{n}\sum\limits_{n_0\leq i\leq n}i^{-1/2+\varepsilon}\Big)+O_{\mathcal{W}}(n^{-1}).
$$
Completing the sum over $i$ in the first term up to an error term bounded by $O_{\mathcal{W}}(n^{-1})$ and using Lemma \ref{lem: bounds for sum over 1<i<i'<n} to estimate the second term, we obtain
\begin{align}\label{eq: (E(S(n))_a,W=(1)}
\mathbb{E}\big(\overline{S}(n)\big)=\frac{1}{n}\sum\limits_{1\leq i\leq n}f_{{\bf b},{\bf u}+{\bf v}}(i)+O_{\alpha,\mathcal{W},\varepsilon}(n^{-1/2+\varepsilon}).
\end{align}
Then our desired result follows from Lemma \ref{lem: mean of f_bu}.
\end{proof}

Now we estimate the variance of $\overline{S}(n)$.
\begin{proposition}\label{prop: V(S(n)_a,W)=}
If $b_1\leq b_2$, then for any $\varepsilon>0$ we have
\begin{align}
\mathbb{V}\big(\overline{S}(n)\big)=O(n^{-1/2+\varepsilon}).
\end{align}
\end{proposition}

\begin{proof}
Write the variance
\begin{align}\label{eq: expression of V(S(n)_a,W)}
\mathbb{V}(\overline{S}(n))=\mathbb{E}\big(\overline{S}(n)^2\big)-\mathbb{E}\big(\overline{S}(n)\big)^2.
\end{align}
It follows from \eqref{eq: (E(S(n))_a,W=(1)} that
\begin{align}\label{eq: E(S(n)_a,W)^2=}
\mathbb{E}\big(\overline{S}(n)\big)^2=\frac{1}{n^2}\bigg(\sum\limits_{1\leq i\leq n}f_{{\bf b},{\bf u}+{\bf v}}(i)\bigg)^2+O_{\mathcal{W},\varepsilon}(n^{-1/2+\varepsilon}).
\end{align}
To deal with $\mathbb{E}\big(\overline{S}(n)^2\big)$, we expand the square and write
\begin{align}\label{eq: expression of E(S(n)_a,W)^2}
\mathbb{E}\big(\overline{S}(n)^2\big)=\frac{2}{n^2}\sum\limits_{1\leq i<i^{\prime}\leq n}\mathbb{E}(X_iX_{i^{\prime}})+{\frac{1}{n^2}}\sum\limits_{1\leq i\leq n}\mathbb{E}(X_i^2).
\end{align}

For $i<i^{\prime}$, we have
$$
\mathbb{E}(X_iX_{i^{\prime}})=\mathbb{P}(P_i~\text{and}~P_{i^{\prime}}~\text{are both {\bf b}-visible to}~\mathcal{W})=:\mathbb{P}_{i,i^{\prime}}^2,
$$
say. Let $n_0>0$ be a constant larger than $\max\limits_{1\leq i\leq J}(|u_j|+1)$ and $\max\limits_{1\leq i\leq J}(|v_j|+1)$, then we have
\begin{align}\label{eq: sum of E(X_iX_i')=}
\sum\limits_{1\leq i<i^{\prime}\leq n}\mathbb{E}(X_i X_{i^{\prime}})=\sum\limits_{n_0\leq i<i^{\prime}\leq n}\mathbb{P}_{i,i^{\prime}}^2+O_{\mathcal{W}}(n)
\end{align}
By a similar argument as that before \eqref{eq: one-point-probability}, we see that the probability such that $P_i=(k,i-k)$ and $P_{i^{\prime}}=(k+l,i^{\prime}-k-l)$ for some $0\leq k\leq i$ and $0\leq l\leq i^{\prime}-i$ is (here it means step $i^{\prime}$ depends on step $i$)
$$
{i\choose k} \alpha^k(1-\alpha)^{i-k}{{i^{\prime}-i}\choose l} \alpha^{l}(1-\alpha)^{i^{\prime}-i-l}.
$$
It follows from the above and remark \ref{remark-equal-coordinate} in Section 2 that 
\begin{align*}
\mathbb{P}_{i,i^{\prime}}^2=\sum\limits_{\substack{0\leq k\leq i\\ \bgcd(k-u_j,i-k-v_j)=1\\1\leq j\leq J}}{i\choose k} \alpha^k(1-\alpha)^{i-k}\sum\limits_{\substack{0\leq l\leq i^{\prime}-i\\ \bgcd(k+l-u_j,i^{\prime}-k-l-v_j)=1\\1\leq j\leq J}}{{i^{\prime}-i}\choose l} \alpha^{l}(1-\alpha)^{i^{\prime}-i-l}
\end{align*}
for $n_0<i<i^{\prime}$. By (ii) of Lemma \ref{lem: properties of bgcd}, the inner sum over $l$ is equal to
$$
\sum\limits_{\substack{0\leq l\leq i^{\prime}-i\\\bgcd(i^{\prime}-u_j-v_j,i^{\prime}-l-k-v_j)=1\\1\leq j\leq J}}{{i^{\prime}-i}\choose l} \alpha^{l}(1-\alpha)^{i^{\prime}-i-l}.
$$
Applying Lemma \ref{lem: mean of f_bu} with $s_j=u_j+v_j$ and $t_j=l-i+k+v_j$, $1\leq j\leq J$, we then have
\begin{align*}
\mathbb{P}_{i,i^{\prime}}^2=\sum\limits_{\substack{0\leq k\leq i\\ \bgcd(k-u_j,i-k-v_j)=1\\1\leq j\leq J}}{i\choose k} \alpha^k(1-\alpha)^{i-k}\Big(f_{{\bf b},{\bf u}+{\bf v}}(i^{\prime})+O_{\alpha,\mathcal{W},\varepsilon}\big(({i^{\prime}}-i)^{-1/2+\varepsilon}\big)\Big)
\end{align*}
for $n_0<i<i^{\prime}$. By (ii) of Lemma \ref{lem: properties of bgcd} and using the binomial theorem to bound the error term, we obtain
\begin{align*}
\mathbb{P}_{i,i^{\prime}}^2=f_{{\bf b},{\bf u}+{\bf v}}(i^{\prime})\sum\limits_{\substack{0\leq k\leq i\\ \bgcd(i-u_j-v_j,i-k-v_j)=1\\1\leq j\leq J}}{i\choose k} \alpha^k(1-\alpha)^{i-k}+O_{\alpha,\mathcal{W},\varepsilon}\big(({i^{\prime}}-i)^{-1/2+\varepsilon}\big)
\end{align*}
for $n_0<i<i^{\prime}$ and any $\varepsilon>0$. This combining Lemma \ref{lem: sum with gcdb condition} yields
\begin{align*}
\mathbb{P}_{i,i^{\prime}}^2=f_{{\bf b},{\bf u}+{\bf v}}(i^{\prime})\big(f_{{\bf b},{\bf u}+{\bf v}}(i)+O_{\alpha,\mathcal{W},\varepsilon}(i^{-1/2+\varepsilon})\big)+O_{\alpha,\mathcal{W},\varepsilon}\big(({i^{\prime}}-i)^{-1/2+\varepsilon}\big),
\end{align*}
which gives
\begin{align}\label{eq: P_2=}
\mathbb{P}_{i,i^{\prime}}^2=f_{{\bf b},{\bf u}+{\bf v}}(i^{\prime})f_{{\bf b},{\bf u}+{\bf v}}(i)+O_{\alpha,\mathcal{W},\varepsilon}({i^{\prime}}^{\varepsilon}i^{-1/2+\varepsilon}+({i^{\prime}}-i)^{-1/2+\varepsilon}\big)
\end{align}
for $n_0<i<i^{\prime}$, where we have used the bound 
\begin{align}\label{eq: bound of f_bs}
|f_{{\bf b},{\bf s}}(n)|\leq \sum_{d_j^{b_1}\mid n-s_j,1\leq j\leq J}\frac{1}{d_1\cdots d_J}=O_{\alpha,\mathcal{W},\varepsilon}(n^{\varepsilon})
\end{align}
for any $\varepsilon>0$ and $n>\max\limits_{1\leq j\leq J}(|s_j|+1)$. It follows from \eqref{eq: sum of E(X_iX_i')=} and \eqref{eq: P_2=} that
$$
\sum\limits_{n_0\leq i<i^{\prime}\leq n}\mathbb{P}_{i,i^{\prime}}^2=\sum\limits_{n_0\leq i<i^{\prime}\leq n}f_{{\bf b},{\bf u}}(i)f_{{\bf b},{\bf u}}(i^{\prime})+O\bigg(n^{\varepsilon}\sum\limits_{n_0\leq i<i^{\prime}\leq n}\Big(i^{-1/2}+(i^{\prime}-i)^{-1/2}\Big)\bigg).
$$
Estimating the $O$-term by Lemma \ref{lem: bounds for sum over 1<i<i'<n}, we obtain
$$
\sum\limits_{n_0\leq i<i^{\prime}\leq n}\mathbb{P}_{i,i^{\prime}}^2=\sum\limits_{n_0<i<i^{\prime}\leq n}f_{{\bf b},{\bf u}}(i)f_{{\bf b},{\bf u}}(i^{\prime})+O(n^{3/2+\varepsilon}).
$$
Complete the sum over $i$ and $i^{\prime}$ up to an error term
$$
\sum\limits_{i\leq n_0}\sum\limits_{i<i^{\prime}\leq n}f_{{\bf b},{\bf u}}(i)f_{{\bf b},{\bf u}}(i^{\prime})=O(n^{1+\varepsilon}),
$$
where we have used bound \eqref{eq: bound of f_bs}, then we derive
$$
\sum\limits_{n_0<i<i^{\prime}\leq n}\mathbb{P}_{i,i^{\prime}}^2=\sum\limits_{1\leq i<i^{\prime}\leq n}f_{{\bf b},{\bf u}}(i)f_{{\bf b},{\bf u}}(i^{\prime})+O(n^{3/2+\varepsilon}).
$$
Adding diagonal term up to an error term which is $O(n^{\varepsilon})$, we have
$$
\sum\limits_{n_0<i<i^{\prime}\leq n}\mathbb{P}_{i,i^{\prime}}^2=\frac{1}{2}\bigg(\sum\limits_{1\leq i\leq n}f_{{\bf b},{\bf u}}(i)\bigg)^2+O(n^{3/2+\varepsilon}).
$$
Inserting this into \eqref{eq: sum of E(X_iX_i')=}, we obtain
\begin{align}\label{eq: sum_iE(X(W)_iX(W)_i')=}
\sum\limits_{1\leq i<i^{\prime}\leq n}\mathbb{E}(X_i X_{i^{\prime}})=\frac{1}{2}\bigg(\sum\limits_{1\leq i\leq n}f_{{\bf b},{\bf u}}(i)\bigg)^2+O(n^{3/2+\varepsilon}).
\end{align}
By the definition of $X_i$, we have
\begin{align}\label{eq: sum_iE(X(W))i^2=}
\sum\limits_{1\leq i\leq n}\mathbb{E}\big(X_i^2\big)=\sum\limits_{1\leq i\leq n}\mathbb{E}(X_i)=O(n).
\end{align}
Combining \eqref{eq: sum_iE(X(W)_iX(W)_i')=}, \eqref{eq: sum_iE(X(W))i^2=} with \eqref{eq: expression of E(S(n)_a,W)^2} gives
$$
\mathbb{E}\big(\overline{S}(n)^2\big)=\frac{1}{n^2}\bigg(\sum\limits_{1\leq i\leq n}f_{{\bf b},{\bf u}+{\bf v}}(i)\bigg)^2+O(n^{-1/2+\varepsilon}).
$$
Inserting this and \eqref{eq: E(S(n)_a,W)^2=} into \eqref{eq: expression of V(S(n)_a,W)} yields our desired result.
\end{proof}

Now Theorem \ref{thm: Theorem for multiple watchpoints} follows from Propositions \ref{prop: (E(S(n))_a,W=}, \ref{prop: V(S(n)_a,W)=} and Lemma \ref{lem: second moment method}.

\section{Proof of Theorem \ref{thm: Theorem for multiple walkers}}
Similar as in the proof of Theorem \ref{thm: Theorem for multiple watchpoints}, we  compute the expectation and variance of $\overline{R}(n)$. As pointed out in Remark \ref{remark-b1-b2}, we only give the proof for the case $b_1\leq b_2$. 
\begin{proposition}\label{prop: E(R(n)_a)=}
If $b_1\leq b_2$, then for any $\varepsilon>0$ we have
\begin{align}
\mathbb{E}\big(\overline{R}(n)\big)=C_{{\bf b},r}+O(n^{-1/2+\varepsilon}),
\end{align}
as $n\rightarrow\infty$, where $C_{{\bf b},r}$ is defined in Lemma \ref{lem: mean of f_b^r(n)}.
\end{proposition}
\begin{proof}
We write
$$
\mathbb{E}\big(\overline{R}(n)\big)=\frac{1}{n}\sum_{1\leq i\leq n} \mathbb{E}\big(Y_i\big)=\frac{1}{n}\sum_{1\leq i\leq n} \mathbb{P}\big(\text{all $P^{(j)}_i,~1\leq j\leq r$ are $\mathbf{b}$-visible}\big).
$$
For simplicity, we denote
\begin{align}\label{eq: definition of P_3}
\mathbb{P}_i^3:=\mathbb{P}\big(\text{all $P^{(j)}_i,~1\leq j\leq r$ are  $\mathbf{b}$-visible}\big).
\end{align}
Then we have
\begin{align}\label{eq: expression of E(R(n)_a)}
\mathbb{E}\big(\overline{R}(n)\big)=\frac{1}{n}\sum_{1< i\leq n}\mathbb{P}_i^3+O(n^{-1}).
\end{align}
and
$$
\mathbb{P}_i^3=\prod\limits_{1\leq j\leq r}\sum\limits_{\substack{0\leq k\leq i\\ \bgcd(k,i-k)=1}}{i \choose k}
\alpha_j^k(1-\alpha_j)^{i-k}
$$
for $i>1$. Applying Lemma \ref{lem: sum with gcdb condition} with $J=1$ and $(u_1,v_1)=(0,0)$, we obtain
$$
\mathbb{P}_i^3=\prod\limits_{1\leq j\leq r}\Big(f_{\bf b}(i)+O_{\alpha_j}\Big(i^{-1/2}\sum\limits_{d^{b_2}\mid i}1\Big)\Big)
$$
for $i>1$. Using the estimate $0<f_{\bf b}(n)<1$ and the bound $\sum\limits_{d^{b_2}\mid i}1\leq \tau(n)=O_{\varepsilon}(n^{\varepsilon})$ for $n\in\mathbb{N}$, we expand the product and derive
$$
\mathbb{P}_i^3=f_{\bf b}^r(i)+O_{{\boldsymbol{\alpha}},r}(i^{-1/2+\varepsilon})
$$
for $i>1$. Inserting this into \eqref{eq: expression of E(R(n)_a)} and applying Lemma \ref{lem: bounds for sum over 1<i<i'<n} to estimate the error term yield
\begin{align}\label{eq: E(R(n)_a)=(1)}
\mathbb{E}\big(\overline{R}(n)\big)=\frac{1}{n}\sum\limits_{1\leq i\leq n}f_{\bf b}(i)^r+O(n^{-1/2+\varepsilon}),
\end{align}
which implies our desired result together with Lemma \ref{lem: mean of f_b^r(n)}.
\end{proof}

Now we estimate the variance of $\overline{R}(n)$.
\begin{proposition}\label{prop: V(R(n)_a)=}Suppose $b_1\leq b_2$, then we have
\begin{align}
\mathbb{V}\big(\overline{R}(n)\big)=O(n^{-1/2+\varepsilon}).
\end{align}
\end{proposition}
\begin{proof}
To compute the variance of $\overline{R}(n)$, we write
\begin{align}\label{eq-Variance-Rn-express}
\mathbb{V}\big(\overline{R}(n)\big)=\mathbb{E}\big(\overline{R}(n)^2 \big)-\mathbb{E}\big(\overline{R}(n)\big)^2.
\end{align}
It follows from \eqref{eq: E(R(n)_a)=(1)} that
\begin{align}\label{eq-Expect-Rn-square}
\mathbb{E}\big(\overline{R}(n)\big)^2=\frac{1}{n^2}\Big(\sum\limits_{1\leq i\leq n}f_{\bf b}(i)^r\Big)^2+O(n^{-1/2+\varepsilon}).
\end{align}
For $\mathbb{E}\big(\overline{R}(n)^2 \big)$, we expand the square and obtain
\begin{align}\label{eq: expression of E(R(n)^2)}
\mathbb{E}\big(\overline{R}(n)^2 \big)=\frac{2}{n^2}\sum\limits_{1\leq i<j\leq n}\mathbb{E}\big(Y_i Y_j\big)+\frac{1}{n^2}\mathbb{E}\Big(\sum\limits_{i\leq n}Y_i^2\Big).
\end{align}
By the definition of $Y_i$, we have
\begin{align}\label{eq: expression of E(X(r)_iE(X(r)_i'))}
\mathbb{E}\big(Y_i Y_{i^{\prime}}\big)= \mathbb{P}(\text{all}~P_i^{(j)}~\text{and}~P_{i^{\prime}}^{(j)},1\leq j\leq r~\text{are {\bf b}-visible})=:\mathbb{P}_{i,i^{\prime}}^4,
\end{align}
say. By similar argument as in the proof of Theorem \ref{thm: Theorem for multiple watchpoints}, we can see that for $1<i<i^{\prime}$ and some $1\leq j\leq r$, the probability such that $P_i^{(j)}=(k,i-k)$ and $P_{i^{\prime}}^{(j)}=(k+l,i^{\prime}-k-l)$ are $\bf b$-visible is
$$
{i\choose k} \alpha_j^k(1-\alpha_j)^{i-k}{{i^{\prime}-i}\choose l} \alpha_j^{l}(1-\alpha_j)^{i^{\prime}-i-l}
$$
for some $0\leq k\leq i$ and $0\leq l\leq i^{\prime}-i$. Hence, we have
$$
\mathbb{P}_{i,i^{\prime}}^4=\prod\limits_{1\leq j\leq r}\sum\limits_{\substack{0\leq k\leq i\\ \bgcd(k,i-k)=1}}{i\choose k} \alpha_j^k(1-\alpha_j)^{i-k}\sum\limits_{\substack{0\leq l\leq i^{\prime}-i\\ \bgcd(k+l,i^{\prime}-k-l)=1}}{{i^{\prime}-i}\choose l} \alpha_j^{l}(1-\alpha_j)^{i^{\prime}-i-l}
$$
for $1<i<i^{\prime}$. By (ii) of Lemma \ref{lem: properties of bgcd} and applying Lemma \ref{lem: sum with gcdb condition} by taking $J=1$, $n=i^{\prime}$, $m=i^{\prime}-i$, $s_1=0$ and $t_1=i-l-k$, the inner sum over $l$ is equal to
$$
\sum\limits_{\substack{0\leq l\leq i^{\prime}-i\\ \bgcd(i^{\prime},i^{\prime}-i+i-l-k)=1}}{{i^{\prime}-i}\choose l} \alpha_j^{l}(1-\alpha_j)^{i^{\prime}-i-l}=f_{\bf b}(i^{\prime})+O_{\alpha_j,{\bf b},\varepsilon}\big((i^{\prime}-i)^{-1/2+\varepsilon}\big)
$$
for $1<i<i^{\prime}$.
It follows that
\begin{align*}
\mathbb{P}_{i,i^{\prime}}^4=&\prod\limits_{1\leq j\leq r}\sum\limits_{\substack{0\leq k\leq i\\ \bgcd(k,i-k)=1}}{i\choose k} \alpha_j^k(1-\alpha_j)^{i-k}\Big(f_{\bf b}(i^{\prime})+O_{\alpha_j,{\bf b},\varepsilon}\big((i^{\prime}-i)^{-1/2+\varepsilon}\big)\Big)\\
=&\prod\limits_{1\leq j\leq r}\bigg(f_{\bf b}(i^{\prime})\sum\limits_{\substack{0\leq k\leq i\\ \bgcd(k,i-k)=1}}{i\choose k} \alpha_j^k(1-\alpha_j)^{i-k}+O_{\alpha_j,{\bf b},\varepsilon}\big((i^{\prime}-i)^{-1/2+\varepsilon}\big)\bigg)
\end{align*}
for $1<i<i^{\prime}$, where we have used the binomial theorem to bound the contribution of the $O$-term. Apply (ii) of Lemma \ref{lem: properties of bgcd} to change the condition $\bgcd(k,i-k)=1$ to $\bgcd(i,i-k)$ and apply Lemma \ref{lem: sum with gcdb condition} with $J=1$, $n=i$, $m=i$, $s_1=0$ and $t_1=k$, then we obtain
\begin{align}\label{eq: P_4=}
\mathbb{P}_{i,i^{\prime}}^4&=\prod\limits_{1\leq j\leq r}\Big(f_{\bf b}(i)f_{\bf b}(i^{\prime})+O_{\boldsymbol{\alpha}_j,\vb,\varepsilon}\big(i^{-1/2+\varepsilon}+(i^{\prime}-i)^{-1/2+\varepsilon}\big)\Big)
\end{align}
for $1<i<i^{\prime}$ by noting $0<f_{\bf b}(n)<1$ for $n\in\mathbb{N}$. Thus by \eqref{eq: expression of E(X(r)_iE(X(r)_i'))} and \eqref{eq: P_4=}, we have
$$
\sum\limits_{1< i<i^{\prime}\leq n}\mathbb{P}_{i,i^{\prime}}^4=\sum\limits_{1< i<i^{\prime}\leq n}\prod\limits_{1\leq j\leq r}\Big(f_{\bf b}(i)f_{\bf b}(i^{\prime})+O_{\boldsymbol{\alpha}_j,\vb,\varepsilon}\big(i^{-1/2+\varepsilon}+(i^{\prime}-i)^{-1/2+\varepsilon}\big)\Big).
$$
Expanding the product and estimating terms containing the $O$-terms, we derive
\begin{align*}
\sum\limits_{1< i<i^{\prime}\leq n}\mathbb{P}_{i,i^{\prime}}^4=\sum\limits_{1< i<i^{\prime}\leq n}f_{\bf b}(i)^r f_{\bf b}(i^{\prime})^r+O_{\boldsymbol{\alpha}_j,\vb,r,\varepsilon}\bigg(\sum\limits_{1<i<i^{\prime}\leq n}\Big(i^{-1/2+\varepsilon}+(i^{\prime}-i)^{-1/2+\varepsilon}\Big)\bigg).
\end{align*}
Estimating the $O$-term by Lemma \ref{lem: bounds for sum over 1<i<i'<n}, we obtain
$$
\sum\limits_{1< i<i^{\prime}\leq n}\mathbb{P}_{i,i^{\prime}}^4=\sum\limits_{1<i<i^{\prime}\leq n}f_{\bf b}(i)^r f_{\bf b}(i^{\prime})^r+O(n^{3/2+\varepsilon}),
$$
which implies
\begin{align}\label{eq: Expect-Xr-cross-term}
\sum\limits_{1\leq i<i^{\prime}\leq n}\mathbb{E}\big(Y_i Y_{i^{\prime}}\big)=\frac{1}{2}\bigg(\sum\limits_{1\leq i\leq n}f_{\bf b}(i)^r\bigg)^2+O(n^{3/2+\varepsilon}).
\end{align}
by adding diagonal terms. By the definition of $Y_i$, we have
\begin{align}\label{eq: sum_iE(Y_i^2)=}
\sum\limits_{1\leq i\leq n}\mathbb{E}\big(Y_i^2\big)=\sum\limits_{1\leq i\leq n}\mathbb{E}(Y_i)=O(n).
\end{align}
Combining \eqref{eq: Expect-Xr-cross-term} and \eqref{eq: sum_iE(Y_i^2)=} with \eqref{eq: expression of E(R(n)^2)}, we obtain
$$
\mathbb{E}\big(\overline{R}(n)^2 \big)=\frac{1}{n^2}\bigg(\sum\limits_{1\leq i\leq n}f_{\bf b}(i)^r\bigg)^2+O(n^{-1/2+\varepsilon}).
$$
Inserting this and \eqref{eq-Expect-Rn-square} into \eqref{eq-Variance-Rn-express} yields our desired result.
\end{proof}

Now Theorem \ref{thm: Theorem for multiple walkers} follows from Propositions \ref{prop: E(R(n)_a)=}, \ref{prop: V(R(n)_a)=} and Lemma \ref{lem: second moment method}.

\section{Sources of main terms}\label{section-lemma-proofs}
In this section, we use tools from number theory to prove Lemmas \ref{lem: sum with gcdb condition}-\ref{lem: mean of f_b^r(n)}
\subsection{Summation with generalized gcd conditions}

In this subsection we give the
proof of Lemma \ref{lem: sum with gcdb condition}

\begin{proof}[Proof of Lemma \ref{lem: sum with gcdb condition}]	For simplicity, we denote
	$$
	{\sum}_{k}=\sum\limits_{\substack{0\leq k\leq m\\ \bgcd(n-s_j,k-t_j)=1, 1\leq j\leq J}}{m\choose k}
	\alpha^k(1-\alpha)^{m-k}.
	$$
	Using the formula
	$$
	\sum_{d\mid n}\mu(d)=
	\begin{cases}
	1,&\text{if}~n=1,\\
	0,&\text{otherwise},
	\end{cases}
	$$
	we may write
	$$
	{\sum}_{k}=\sum\limits_{0\leq k\leq m}{m\choose k}
	\alpha^k(1-\alpha)^{m-k}\prod\limits_{1\leq j\leq J}\sum_{d_j\mid \bgcd(n-s_j,k-t_j)}\mu(d).
	$$
	By (i) of Lemma \ref{lem: properties of bgcd} and changing the order of summations, we derive
	\begin{align}\label{eq: sum with bgcd conditions}
	{\sum}_{k}=\sum_{\substack{d_j^{b_1}\mid n-s_j, 1\leq j\leq J}}\mu(d_1)\cdots\mu(d_J)\sum\limits_{\substack{0\leq k\leq m\\ k\equiv t_j\bmod d_j^{b_2}, 1\leq j\leq J}}{m\choose k}
	\alpha^k(1-\alpha)^{m-k}.
	\end{align}
	We analyze the conditions of the sums on the right hand side. For any $1\leq j_1\neq j_2\leq J$, letting $d=\gcd(d_{j_1},d_{j_2})$, we then have $d^{b_1}\mid n-s_{j_1}$ and $d^{b_1}\mid n-s_{j_2}$, which gives $d^{b_1}\mid s_{j_1}-s_{j_2}$. Similarly, we have $d^{b_2}\mid t_{j_1}-t_{j_2}$. It follows from (i) of Lemma \ref{lem: properties of bgcd} that $d\mid \bgcd(s_{j_1}-s_{j_2},t_{j_1}-t_{j_2})=1$, which gives $d=1$. This indicates $d_1,\cdots,d_J$ are pairwise coprime to each other. Then by the Chinese Reminder Theorem and Lemma \ref{lem: sum with congruence condition}, we have
	$$
	\sum\limits_{\substack{0\leq k\leq m\\ k\equiv t_j\bmod d_j^{b_2}, 1\leq j\leq J}}{m\choose k}
	\alpha^k(1-\alpha)^{m-k}=\frac{1}{(d_1\cdots d_J)^{b_2}}+O_{\alpha}(m^{-\frac{1}{2}}).
	$$
	Inserting this into \eqref{eq: sum with bgcd conditions} and using the bound $|\mu(d)|\leq 1$ for any $d\in\mathbb{N}$ to estimate the error term, we obtain
	$$
	{\sum}_{k}=f_{{\bf b},{\bf s}}(n)+O_{\alpha}\bigg(m^{-1/2}\prod\limits_{1\leq j\leq J}\sum\limits_{d^{b_1}\mid n-s_j}1\bigg).
	$$
	Then our desired result follows from the bound $\sum\limits_{d^{b_1}\mid n}1\leq\tau(n)=O_{\varepsilon}(n^{\varepsilon})$ for $n\geq 1$.
\end{proof}

\subsection{Source of the main term}In this subsection, we give the proof of Lemma \ref{lem: mean of f_bu}.

\begin{proof}[Proof of Lemma \ref{lem: mean of f_bu}]
	Let $s_0=\max\limits_{1\leq j\leq J}|s_j|$. Since the contribution of those $n\leq s_0$ is $O_{\vb,{\bf s}}(1)$, then we only need to cosider the case $s_0<n\leq x$. By the definition of $f_{{\bf b},{\bf s}}$, we write
	$$
	\sum\limits_{s_0<n\leq x}f_{{\bf b},{\bf s}}(n)=\sum\limits_{s_0<n\leq x}\sum_{\substack{d_j^{b_1}\mid n-s_j,1\leq j\leq J\\ \gcd(d_{j_1},d_{j_2})=1,\forall 1\leq j_1\neq j_2\leq J}}\frac{\mu(d_1)\cdots\mu(d_J)}{(d_1\cdots d_J)^{b_2}}.
	$$
	Changing the order of summations, we have
	\begin{align*}
	\sum\limits_{s_0<n\leq x}f_{{\bf b},{\bf s}}(n)&=\sum_{\substack{d_j\leq (x+|s_j|)^{1/b_1},1\leq j\leq J\\ \gcd(d_{j_1},d_{j_2})=1,\forall 1\leq j_1\neq j_2\leq J}}\frac{\mu(d_1)\cdots\mu(d_J)}{(d_1\cdots d_J)^{b_2}}\sum\limits_{\substack{s_0<n\leq x\\ n\equiv s_j\bmod d_j^{b_1},1\leq j\leq J}}1\\
	&=\sum_{\substack{d_j\leq (x+|s_j|)^{1/b_1},1\leq j\leq J\\ \gcd(d_{j_1},d_{j_2})=1,\forall 1\leq j_1\neq j_2\leq J}}\frac{\mu(d_1)\cdots\mu(d_J)}{(d_1\cdots d_J)^{b_2}}\bigg(\frac{x}{(d_1\cdots d_J)^{b_1}}+O_{{\bf b},{\bf s}}(1)\bigg),
	\end{align*}
	which implies
	$$
	\sum\limits_{s_0<n\leq x}f_{{\bf b},{\bf s}}(n)=x\sum_{\substack{d_j\leq (x+|s_j|)^{1/b_1},1\leq j\leq J\\ \gcd(d_{j_1},d_{j_2})=1,\forall 1\leq j_1\neq j_2\leq J}}\frac{\mu(d_1)\cdots \mu(d_J)}{(d_1\cdots d_J)^{b_1+b_2}}+O_{{\bf b},{\bf s}}\big(\log^{J} x\big).
	$$
	Extending the range of $d_j,1\leq j\leq J$ to all positive integers, the error occurs is $O_{\vb,{\bf s}}( x^{1-b_2/b_1})$, which can be absorbed since $b_1\leq b_2$. Hence we have
	$$
	\sum\limits_{s_0<n\leq x}f_{{\bf b},{\bf s}}(n)=x\sum_{\substack{d_1,\cdots,d_J\in\mathbb{N}\\ \gcd(d_{j_1},d_{j_2})=1,\forall 1\leq j_1\neq j_2\leq J}}\frac{\mu(d_1)\cdots \mu(d_J)}{(d_1\cdots d_J)^{b_1+b_2}}+O_{{\bf b},{\bf s}}\big(\log^{J} x\big).
	$$
	Letting $d=d_1\cdots d_J$, we then obtain
	$$
	\sum\limits_{s_0<n\leq x}f_{{\bf b},{\bf s}}(n)=x\sum_{d=1}^{\infty}\frac{\mu(d)\tau_{J}(n)}{d^{b_1+b_2}}+O_{{\bf b},{\bf s}}\big(\log^{J} x\big).
	$$
	where $\tau_J(n)=\sum\limits_{n=d_1\cdots d_J}1$ is the $J$-fold divisor function. This gives our desired result by noting
	$$
	\sum_{d=1}^{\infty}\frac{\mu(d)\tau_{J}(n)}{d^{b_1+b_2}}=\prod\limits_{p}\bigg(1-\frac{J}{p^{b_1+b_2}}\bigg),
	$$
	where $p$ runs over all primes.
\end{proof}

\subsection{Main term for multiple walkers} In this subsection, we give the proof of Lemma \ref{lem: mean of f_b^r(n)} using analytic methods for Dirichlet series. 

\begin{proof}[Proof of Lemma \ref{lem: mean of f_b^r(n)}]
	By the Euler product formula, the Dirichlet series of $f_{\bf b}(n)^r$ is 
	$$
	\sum_{n=1}^{\infty}\frac{f_{\bf b}(n)^r}{n^s}=\prod_p \bigg(1+\frac{1}{p^s}+\cdots+\frac{1}{p^{(b_1-1)s}}+\frac{(1-p^{-b_2})^r}{p^{b_1s}}+ \frac{(1-p^{-b_2})^r}{p^{(b_1+1)s}}+\cdots \bigg),
	$$
	where $s$ is a complex number with $\Re(s)>1$. Using the Euler product of $\zeta(s)$, we write
	\begin{align}\label{eq-f-Dirichlet-series-product}
	\sum_{n=1}^{\infty}\frac{f_{\bf b}(n)^r}{n^s}=\zeta(s)G(s),
	\end{align}
	where
	$$
	G(s):=G_{{\bf b}, r}(s)=\prod_p \bigg(1-\frac{1}{p^{b_1s}}+\frac{1}{p^{b_1s}}\Big(1-\frac{1}{p^{b_2}}\Big)^r\bigg).
	$$
	Since  $1\leq b_1\leq b_2$, then for any $\varepsilon>0$ and $\Re(s)\geq\varepsilon$, we have
	$$
	\log\bigg(1-\frac{1}{p^{b_1s}}+\frac{1}{p^{b_1s}}\Big(1-\frac{1}{p^{b_2}}\Big)^r\bigg)=\log \bigg(1+O_r\Big(\frac{1}{p^{b_1\Re s+b_2}}\Big)\bigg)=O_r\Big(\frac{1}{p^{b_1\Re s+b_2}}\Big),
	$$
	which gives
	$$
	|G(s)|\leq\exp\bigg(O_r\Big(\sum\limits_{p}\frac{1}{p^{b_1\Re s+b_2}}\Big)\bigg)<\infty.
	$$
	Thus for any $\varepsilon>0$ the product $G(s)$ is absolutely and uniformly convergent in the range $\Re(s)\geq\varepsilon$ and satisfies $G(s)=O_{\varepsilon,{\bf b},r}(1)$. 
	Denote the Dirichlet series of $G(s)$ as $$G(s)=\sum\limits_{n=1}^{\infty}\frac{g_{\vb,r}(n)}{n^s},
	$$
	then we have $g_{{\bf b},r}(n)=O_{\varepsilon,r}(n^{\varepsilon})$ for any $\varepsilon>0$.
	Applying Perron's formula (see e.g. Heath-Brown's notes on Titchmarsh [31], p.70)), we derive  
	$$
	\sum\limits_{n\leq x}g_{\vb,r}(n)=\int_{1+\varepsilon-ix}^{1+\varepsilon+ix}G(s)\frac{x^s}s{\rm d}s+O_{\varepsilon,{\bf b},r}(x^{\varepsilon})
	$$
	for any $\varepsilon>0$. Shifting the integral contour by the residue theorem, we have
	$$
	\sum\limits_{n\leq x}g_{\vb,r}(n)=\bigg(\int_{\varepsilon-ix}^{\varepsilon+ix}+\int_{1+\varepsilon-ix}^{\varepsilon-ix}+\int_{\varepsilon+ix}^{1+\varepsilon+ix}\bigg)G(s)\frac{x^s}s{\rm d}s+O_{\varepsilon,{\bf b},r}(x^{\varepsilon}).
	$$
	Using the upper bound of $G(s)$ to estimate the integrals, we obtain
	\begin{align}\label{eq-sum-g(n)=}
	\sum\limits_{n\leq x}g_{\vb,r}(n)=O_{\varepsilon,{\bf b},r}(x^{\varepsilon}).
	\end{align}
	Now by \eqref{eq-f-Dirichlet-series-product}, we have the relation
	$$
	f_{\bf b}(n)^r=\sum\limits_{kl=n}g_{\vb,r}(k).
	$$
	Thus write
	$$
	\sum\limits_{n\leq x}f_{\bf b}(n)^r=\sum\limits_{k\leq \sqrt{x}}g_{\vb,r}(k)\sum\limits_{l\leq x/k}1+\sum\limits_{l\leq \sqrt{x}}\sum\limits_{k\leq x/l}g_{\vb,r}(k)-\Big(\sum\limits_{k\leq\sqrt{x}}g_{\vb,r}(k)\Big)\Big(\sum\limits_{l\leq \sqrt{x}}1\Big),
	$$
	then with the help of \eqref{eq-sum-g(n)=} we have
	$$
	\sum\limits_{l\leq \sqrt{x}}\sum\limits_{k\leq x/l}g_{\vb,r}(k)=\sum\limits_{l\leq \sqrt{x}}O_{\varepsilon,{\bf b},r}(x^{\varepsilon})=O_{\varepsilon,{\bf b},r}(x^{1/2+\varepsilon}).
	$$
   and
	$$
	\Big(\sum\limits_{k\leq\sqrt{x}}g_{\vb,r}(k)\Big)\Big(\sum\limits_{l\leq \sqrt{x}}1\Big)=O_{\varepsilon,{\bf b},r}(x^{1/2+\varepsilon}).
	$$
	Hence by the estimate $g_{\vb,r}(n)=O(n^{\varepsilon})$, we have
	\begin{align}\label{eq-sum-f(n)=}
	\sum\limits_{n\leq x}f_{\bf b}(n)^r=&\sum\limits_{k\leq \sqrt{x}}g_{\vb,r}(k)\Big(\frac{x}{k}+O(1)\Big)+O_{\varepsilon,{\bf b},r}(x^{1/2+\varepsilon})\\
	=& x\sum\limits_{k\leq \sqrt{x}}\frac{g_{\vb,r}(k)}{k}+O_{\varepsilon,{\bf b},r}(x^{1/2+\varepsilon}).\nonumber
	\end{align}
	Note that
	$$
	\sum\limits_{k\leq \sqrt{x}}\frac{g_{\vb,r}(k)}{k}=\sum\limits_{k=1}^{\infty}\frac{g_{\vb,r}(k)}{k}-\sum\limits_{k>\sqrt{x}}\frac{g_{\vb,r}(k)}{k}=G(1)+O_{\varepsilon,{\bf b},r}(x^{-1/2+\varepsilon}).
	$$
	Inserting this into \eqref{eq-sum-f(n)=} yields
	$$
	\sum\limits_{n\leq x}f_{\bf b}(n)^r=G(1)x+O_{\varepsilon,{\bf b},r}(x^{1/2+\varepsilon}).
	$$
	Now we finish our proof.
\end{proof}


\begin{thebibliography}{99}
	
\bibitem{BEH} C. Benedetti, S. Estupi\~{n}\'{a}n, P. E. Harris, Generalized Lattice Point Visibility, \textit{Preprint
			available at} https://arxiv.org/abs/2001.07826.			
	
	
		
\bibitem{BCZ} F. P. Boca, C. Cobeli, A. Zaharescu. Distribution of Lattice Points Visible from the Origin. \textit{Communications in Mathematical Physics}, \textbf{213:2} (2000), 433-470.

\bibitem{CZ}S. Chaubey, A. Tamazyan, A. Zaharescu, Lattice point problems involving index and joint visibility. \textit{Proc. Amer. Math. Soc.} \textbf{147:8} (2019), 3273-3288.

\bibitem{CFF}J. Cilleruelo, J. L. Fern\'{a}ndez, P. Fern\'{a}ndez, Visible lattice points in random walks. \textit{European Journal of Combinatorics} {\bf 75} (2019) 92-112.
		
\bibitem{GHKM} E. H. Goins, P. E. Harris, B. Kubik, A. Mbirika, Lattice Point Visibility on Generalized Lines of Sight, \textit{The American Mathematical Monthly}, \textbf{125:7} (2018), 593-601.
		
\bibitem{HO} P. E. Harris, M. Omar, Lattice point visibility on power functions, \textit{Integers}, \textbf{18} (2018), A90, 1-7.

\bibitem{IK} H. Iwaniec,  E. Kowalski,  \textit{Analytic Number Theory}, vol. \textbf{53}. Colloquium Publications, American Mathematical Society, Providence (2004).

\bibitem{LM} K. Liu,  X. Meng, Visible lattice points along curves, \textit{The Ramanujan Journal} (2020), https://doi.org/10.1007/s11139-020-00302-w.
		
\bibitem{R-thesis} D. F.  Rearick,  Some visibility problems in point lattices. Dissertation (Ph.D.)(1960), California Institute of Technology. http://resolver.caltech.edu/CaltechETD: etd-06232006-133908
		
\bibitem{R} D. F.  Rearick, Mutually visible lattice points, \textit{Norske Vid. Selsk. Forh.} (Trondheim) \textbf{39} (1966), 41-45.
		
\bibitem{S} J. J. Sylvester, Sur le nombre de fractions ordinaires inegales quon peut exprimer en se	servant de chiffres qui nexcedent pas un nombre donne, C. R. Acad. Sci. Paris XCVI (1883), 409–413. Reprinted in H.F. Baker (Ed.), \textit{The Collected Mathematical Papers of	James Joseph Sylvester}, vol. \textbf{4}, Cambridge University Press, p. 86.

\bibitem{TH} E. C. Titchmarsh, revised by D. R.  Heath-Brown, The Theory of the Riemann zeta-function,
2nd ed., Clarendon Press, Oxford, 1986.
	
\end{thebibliography}
\end{document}